\documentclass[preprint,3p,12pt,pdf]{elsarticle}

\usepackage{lineno,hyperref}
\modulolinenumbers[5]

\usepackage{hyperref}

\usepackage{epstopdf}

\usepackage{tikz}
\usetikzlibrary{arrows}

\usepackage{pst-node}
\usepackage{tikz-cd}

\usepackage[shellescape]{gmp}

\usepackage{mathrsfs}
\usepackage{amssymb}
\usepackage{amsfonts}
\usepackage{latexsym}
\usepackage{mathtools}
\usepackage{xcolor}
\usepackage{wrapfig}
\usepackage{floatflt}

\usepackage{mathtools}
\usepackage{extarrows}

\usepackage{graphicx}

\usepackage{subcaption}

\def\lb{\label}

\newcommand{\er}[1]{\textrm{(\ref{#1})}}

\newtheorem{theorem}{\bf Theorem}[section]
\newtheorem{lemma}[theorem]{\bf Lemma}

\def\a{\alpha}         
\def\b{\beta}          
\def\g{\gamma}         
\def\G{\Gamma}         
\def\d{\delta}

    \def\cJ{{\mathcal J}}       \def\mJ{{\mathscr J}}
        
\def\k{\kappa}

            \def\mP{{\mathscr P}}

    \def\cT{{\mathcal T}}

\def\ve{\varepsilon}   \def\vt{\vartheta}    \def\vp{\varphi}    

\def\Z{{\mathbb Z}}    \def\R{{\mathbb R}}   \def\C{{\mathbb C}}    
    \def\N{{\mathbb N}}   

\def\lt{\biggl}                  \def\rt{\biggr}
               \def\wt{\widetilde}
\def\no{\noindent}

\let\ge\geqslant                 \let\le\leqslant

\def\iy{\infty}
\def\sm{\setminus}               
\def\ss{\subset}                 \def\ts{\times}
\def\pa{\partial}

\def\el2{\ell^{\,2}}             \def\1{1\!\!1}

\def\arg{\mathop{\mathrm{arg}}\nolimits}

\def\Im{\mathop{\mathrm{Im}}\nolimits}

\def\Re{\mathop{\mathrm{Re}}\nolimits}

\def\sign{\mathop{\mathrm{sign}}\nolimits}

\def\BBox{\hspace{1mm}\vrule height6pt width5.5pt depth0pt \hspace{6pt}}

\newtheorem{corollary}[theorem]{\bf Corollary}

\let\ge\geqslant
\let\le\leqslant

\newcommand{\ca}{\begin{cases}}
	\newcommand{\ac}{\end{cases}}
\newcommand{\ma}{\begin{pmatrix}}
	\newcommand{\am}{\end{pmatrix}}
\renewcommand{\[}{\begin{equation}}
	\renewcommand{\]}{\end{equation}}
\def\eq{\begin{equation}}
	\def\qe{\end{equation}}

\def\BBox{\hspace{1mm}\vrule height6pt width5.5pt depth0pt \hspace{6pt}}

\begin{document}
	
	\begin{frontmatter}

		\title{Left-tail expansions for Schr\"oder branching processes with explicit convergence rates}

		\date{\today}

		\author
		{Anton A. Kutsenko}
	
	\address{University of Hamburg, MIN Faculty, Department of Mathematics, 20146 Hamburg, Germany; email: akucenko@gmail.com}
	
\begin{abstract}
In previous work, the density of the martingale limit in Schr\"oder branching processes was expressed as a convergent double power-law series with oscillatory terms. The proof relied on two assumptions, one of which imposed a geometric restriction on the critical angle of the Julia set of the offspring generating function near $1$. In this paper, we show that both assumptions can be removed.

We derive explicit bounds on the expansion coefficients, which imply locally uniform convergence of the double series. The coefficients decay exponentially in one summation index, with the rate explicitly determined by the critical angle, and super-exponentially in the other.

Finally, we investigate the behavior of the critical angle and describe regimes in which it approaches its minimum. This analysis shows how the geometry of the Julia set influences the magnitude of the oscillatory corrections in the left-tail expansion.
\end{abstract}


	
\end{frontmatter}


{\section{Introduction}\lb{sec0}}
We consider a simple Galton-Watson branching process 
$$
 X_{t+1}=\sum_{j=1}^{X_t}\xi_{j,t},\ \ \ X_0=1,\ \ \ t\in\N\cup\{0\}
$$
in the supercritical case with the minimum family size $1$ - the so-called Schr\"oder case. The probability of the minimum family size is $0<p<1$. Thus, the probability-generating function has the form
\[\lb{001}
 P(z):=\mathbb{E}z^{\xi}=pz+p_2z^2+p_3z^3+....
\]
All $p_j\ge0$ and $P(1)=1$. Thus, $P(z)$ is analytic inside the unit disc $\mathbb{B}_1=\{z:\ |z|<1\}$. We also require that it be analytic in some neighborhood of $z=1$. In particular, the expectation
\[\lb{002}
 E=p+2p_2+3p_3+...<+\iy.
\]
Then one may define the {\it martingale limit} $W=\lim_{t\to+\iy}E^{-t}X_t$, the density of which can be expressed as a Fourier transform of some special function
\[\lb{003}
 w(x)=\frac1{2\pi}\int_{-\iy}^{+\iy}\Pi(\mathbf{i}y)e^{\mathbf{i}yx}dy,
\ \ \ {\rm where}\ \ \ 
 \Pi(z)=\lim_{t\to+\iy}\underbrace{P\circ...\circ P}_{t}(1-\frac{z}{E^{t}}),
\]
see, e.g., \cite{D1}. The function $\Pi(z)$ satisfies Poincar\'e-type functional equation 
\[\lb{004}
 \Pi(Ez)=P(\Pi(z)),\ \ \ \Pi(0)=1,\ \ \ \Pi'(0)=-1.
\]
The properties of Poincar\'e-type functions are well studied, see, e.g., \cite{M}. For convenience, we provide the related results with proofs in this manuscript. The existence and some properties of $\Pi(z)$ are discussed in Theorem \ref{T1}, Corollary \ref{C1}, and Theorem \ref{T3}. The existence (convergence) of the integral in \er{003} is discussed in the proof of the main Theorem. Since the Fourier integral is quite complex, any expansion and asymptotic analysis of $w(x)$ is welcome. Special attention is paid to the analysis of the tails as $x\to+0$ or $x\to+\iy$. It is proven in \cite{BB1} that $w(x)$ has an asymptotic
\[\lb{000}
 w(x)=x^{\a}V(x)+o(x^{\a}),\ \ \ x\to+0,
\]
with explicit $\a=-1-\ln p/\ln E=-1-\log_Ep$ and a continuous, positive, multiplicatively periodic function, $V$, with period $E$. Further references to this asymptotic are always based on the principal work \cite{BB1} and do not provide any formula for $V(x)$, see, e.g., the corresponding remark in \cite{FW2007}, \cite{FW} and \cite{S} devoted to the Schr\"oder case. Even such a result is already great because $w(x)$ is not simple. In \cite{K24}, under some assumptions, the complete expansion 
\[\lb{005}
w(x)=x^{\a}V_1(x)+x^{\a+\b}V_2(x)+x^{\a+2\b}V_3(x)+...,\ \ \ x>0
\]
is derived. Here, the certain value $\a$ is defined above and $\b=-\log_Ep>0$.  Moreover, 
\[\lb{006}
 V_n(x)=K_n(-\log_Ex),\ \ \ K_n(z)=\sum_{m=-\iy}^{+\iy}\frac{\phi_n\k_{m}^{\ast n}e^{2\pi\mathbf{i}mz}}{\Gamma(-\frac{2\pi\mathbf{i}m+n\ln p_1}{\ln E})},\ \ \k_{m}^{\ast n}=\int_0^1K(x)^ne^{-2\pi\mathbf{i}mx}dx,
\]
where
\[\lb{007}
 K(z)=\Phi(\Pi(E^z))p^{-z},\ \ \Phi(z)=\lim_{t\to\iy}p^{-t}\underbrace{P\circ...\circ P}_{t}(z),\ \ \Phi^{-1}(z)=\phi_1z+\phi_2z^2+\phi_3z^3+....
\]
The function $\Phi(z)$ satisfies the Schr\"oder functional equation
\[\lb{008}
 \Phi(P(z))=p\Phi(z),\ \ \ \Phi(0)=0,\ \ \ \Phi'(0)=1.
\]
It can be extended to the open connected component $\mJ_P(0)$ of the filled Julia set containing the unit disc $\mathbb{B}_1$ and having a unique attracting point $z=0$.  Schr\"oder-type functions are well known. Their study begins with work \cite{S1870} and continues in works \cite{K1884}, \cite{P1890}, and \cite{F1}. An excellent exposition of many ideas in holomorphic dynamics is given in a nice book \cite{M}. For convenience, we include some related results with proofs in this manuscript, see Theorem \ref{T4}. The inverse function satisfies a Poincar\'e-type functional equation
\[\lb{008a}
 \Phi^{-1}(pz)=P(\Phi^{-1}(z)),\ \ \ \Phi^{-1}(0)=0,\ \ \ (\Phi^{-1})'(0)=1.
\]
Differentiating \er{008a} and using expansion \er{001}, we can find $\phi_j$ in \er{007} explicitly, step by step
\[\lb{008b}
 \phi_1=1,\ \ \ \phi_2=\frac{p_2}{p^2-p},\ \ \ \phi_3=\frac{2p_2\phi_2+p_3}{p^3-p},\ ....
\] 
One-periodic, as it is seen from \er{007} and functional equations \er{004} and \er{008}, Karlin-McGregor function $K(z)$ was introduced in \cite{KM1} and \cite{KM2}. Equation \er{005} is effective for calculating the density $w(x)$. In \cite{K24}, it is shown that the double series, upon substituting \er{006} into \er{005}, converges superexponentially fast in $n$ and exponentially fast in $m$. However, for the convergence in $m$, the condition on the Julia set critical angle $\vt$ was required. This angle is half of the maximal angle for which there is a sector of arbitrarily small radius and with vertex at $1$ lying entirely in $\mJ_P(0)$. The complement to the doubled critical angle is illustrated in Fig. \ref{fig0}, taken from my article \cite{K}.
\begin{figure}[h]
	\center{\includegraphics[width=0.75\linewidth]{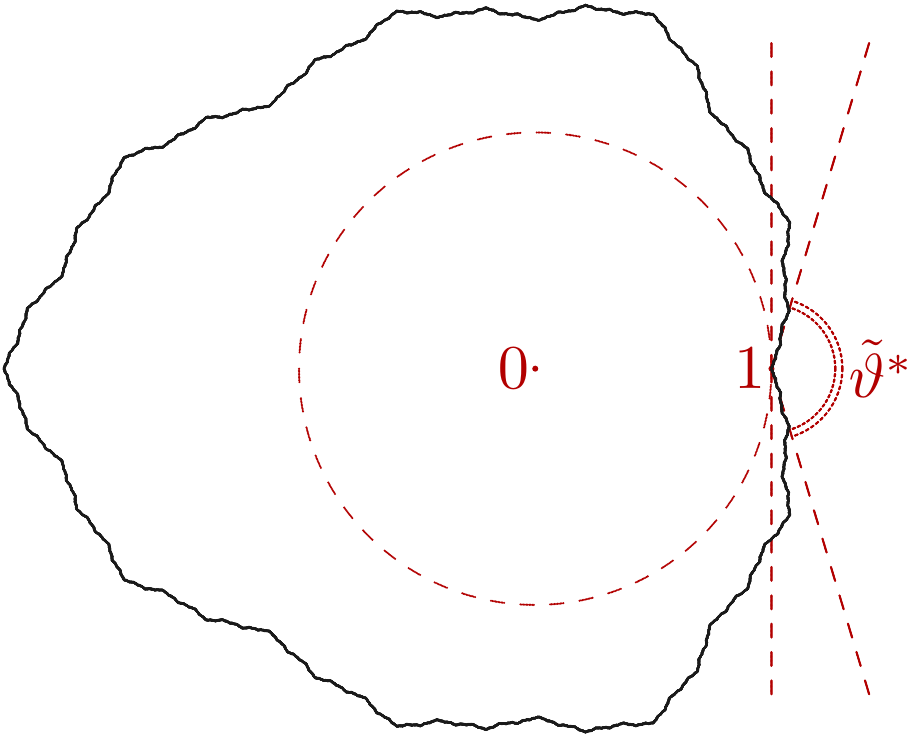}}
	\caption{Julia set for the polynomial $P(z)=0.1z+0.5z^2+0.4z^3$. It is the boundary of the open filled component $\mJ_P(0)$, which is the domain of definition for the analytic function $\Phi$, see \er{007} and \er{008}. The unit disc belongs to the filled Julia set. The critical angle is $\vt=\pi-\tilde\vt^*/2$.}\lb{fig0}
\end{figure}
\begin{theorem}\lb{T24}{\rm [from \cite{K24}, 2024]}
If $\log_Ep<-1$ and $\vt>\pi/2$ then \er{005} is true.
\end{theorem}
The condition $\log_Ep<-1$ can be removed with one trick based on integration by parts. The second condition needs to be addressed differently. Since $\mathbb{B}_1\ss\mJ_P(0)$, we have $\vt\ge\pi/2$. But, to compensate for small values of $\G$-function in the denominator, see \er{006}, we need  $\vt>\pi/2$. In the current manuscript, it is shown that the critical angle $\vt$ is always strictly greater than $\pi/2$ (see Corollary \ref{C2}). The only assumption is $0<p<1$, which is assumed at the beginning. We will formulate a strengthened version of Theorem \ref{T24} in a slightly expanded form, indicating the rate of convergence of the double series.
\begin{theorem}\lb{mainT} The following expansion holds
\[\lb{009}
w(x)=x^{-1}\sum_{n=1}^{+\iy}\sum_{m=-\iy}^{+\iy}\frac{\phi_n\k_{m}^{\ast n}x^{-\frac{n\ln p+2\pi\mathbf{i}m}{\ln E}}}{\Gamma(-\frac{2\pi\mathbf{i}m+n\ln p_1}{\ln E})}
\]
and
\[\lb{010}
 \lt|\frac{\phi_n\k_m^{*n}x^{-\frac{n\ln p+2\pi\mathbf{i}m}{\ln E}}}{\G(-\frac{n\ln p+2\pi\mathbf{i}m}{\ln E})}\rt|\le Ae^{({\log_Ep}+\ve)n\ln n+\frac{2\pi(\frac{\pi}2-\vt+\ve)}{\ln E}|m|},
\]
where, for any fixed $\ve>0$, the value $A=A(\ve,x)>0$ is uniformly bounded when $x$ belongs to any compact subset of $(0,+\iy)$. 
\end{theorem}
As follows from the proof, the exponential factor in \er{010} cannot be significantly improved. The series \er{009} converges absolutely. The convergence is exponential in $m$ and superexponential in $n$. We can change the order of summation by Riemann’s rearrangement theorem and obtain the direct strengthening of Theorem \ref{T24}.
\begin{corollary}\lb{mainC}
Identity \er{005} is true.
\end{corollary}

{\bf Remark 1.} Let us discuss briefly when the critical angle $\vt$ can be minimal. This discussion is not rigorous and should be taken with a grain of salt. However, it is a useful practical illustration. If $P$ is a polynomial with $\deg P=d$ then the B\"otcher function is defined as
\[\lb{b001}
 B(z)=\underbrace{P\circ...\circ P}_{t}(z)^{\frac1{d^t}},\ \ \ z\in\C\sm\mJ_P.
\]
It satisfies the functional equation
\[\lb{b002}
 B(P(z))=B(z)^d,\ \ \ B(z)=p_d^{\frac1{d-1}}z+O(1),\ z\to\iy.
\]
By analogy with the Karlin-McGregor function, one may define another one-periodic function
\[\lb{b003}
 L(z):=d^{-z}\ln B(\Pi(-E^z)),
\]
which, for $z$ in sufficiently small neighborhood of $1$, gives
\[\lb{b004}
 B(z)=\exp(L(\log_E(-\Pi^{-1}(z)))(-\Pi^{-1}(z))^{\log_Ed}),
\]
which, by \er{004}, leads to
\[\lb{b005}
 B(1+z)=1+L(\log_Ez)z^{\log_Ed}+o(1),\ \ \ z\to0.
\]
Definition \er{b001} means that $B(\pa\mJ_P)\ss\mathbb{S}_1=\{z:\ |z|=1\}$, since $\pa\mJ_P$ is bounded for polynomials and invariant under action of $P$. Thus, excluding the effect of the log-periodic factor $L(\log_Ez)$ in \er{b005}, we may see that the mapping $z^{\log_Ed}$ expands, in the first approximation, the complement of $\vt$ to $\pi/2$. Hence,
\[\lb{b006}
 \vt\approx\pi-\frac{\pi}{2\log_Ed},
\]
and the minimal $d$ minimizes the angle $\vt$, in the first approximation. For fixed $p$ and $E$, the minimal $d$ is reached for the optimal $P(z)$ given in Corollary \ref{C0}. This is a $2$ or $3$ point probability generating function, which also minimizes the angle of deviation of the imaginary axis under the action of the mapping $\Pi(z)$, see Lemma \ref{L2} and proof of Theorem \ref{T3}. This deviation is the main detail on which the proof that $\vt>\pi/2$ is based. Thus, from both points of view, the optimal polynomial gives a small critical angle. We have no proof that it gives the smallest angle. A comparison of two critical angles, one of which is empirically minimal, is illustrated in Fig. \ref{fig1}.
\begin{figure}[h]
    \centering
    \begin{subfigure}[b]{0.55\textwidth}
        \includegraphics[width=\textwidth]{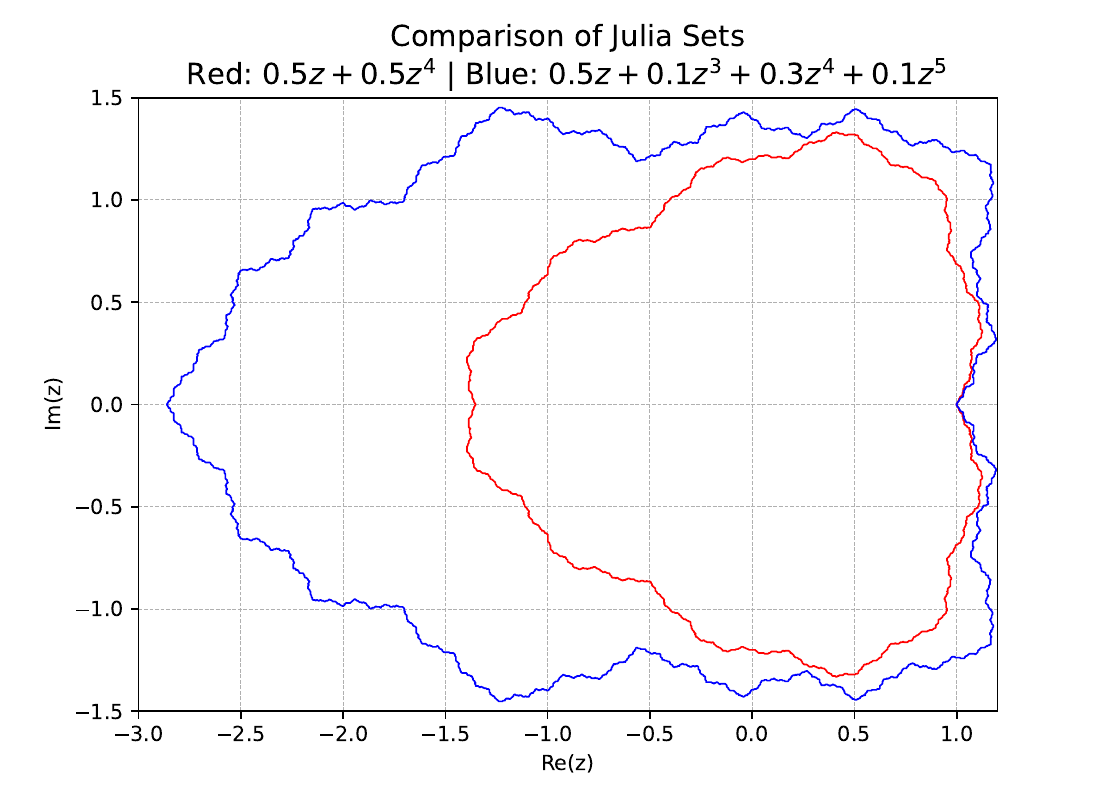}
        \caption{}
    \end{subfigure}
    \begin{subfigure}[b]{0.43\textwidth}
        \includegraphics[width=\textwidth]{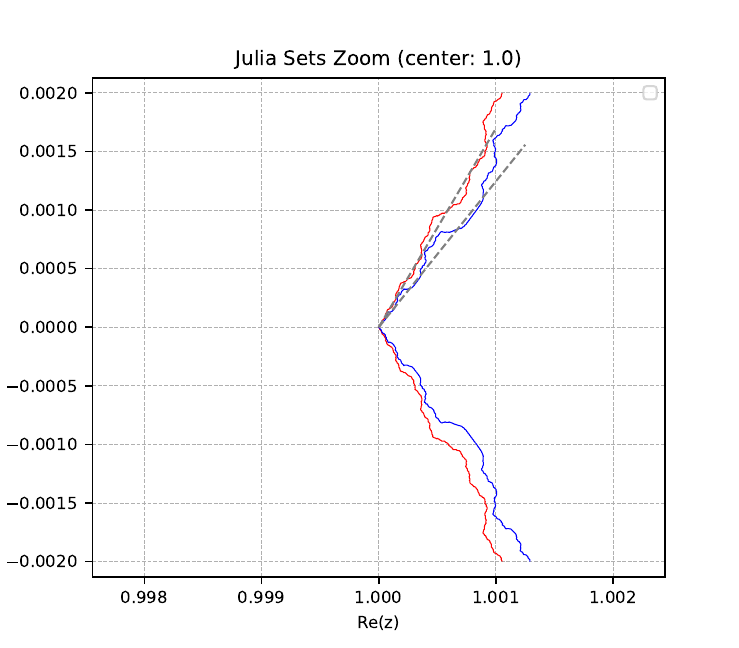}
        \caption{}
    \end{subfigure}
    \caption{Comparison of the $0$-components of Julia sets for two polynomials having the same $p=0.5$ and $E=2.5$, where (a) provides a general overview, and (b) gives a zoom at $z=1$. Dashed lines in (b) correspond to the RHS of \er{b006}.}\lb{fig1}
\end{figure}
The factor related to $|m|$ in the estimate \er{010} shows that smaller critical angles can lead to larger density oscillations. This statement is not rigorous. Many factors must be taken into account, especially the value of $A$. For a particular example of polynomials considered in Fig. \ref{fig1}, the non-rigorous statement about oscillations is true, see Fig. \ref{fig2}. The corresponding densities without normalization are compared in Fig. \ref{fig2a}.
\begin{figure}[h]
	\center{\includegraphics[width=0.95\linewidth]{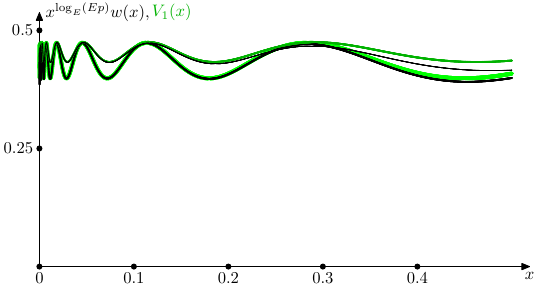}}
	\caption{For the probability generating functions taken from Fig. \ref{fig1}, the first oscillatory term is compared with the normalized density. Larger oscillations correspond to the empirically optimal polynomial $P(z)=(z+z^4)/2$.}\lb{fig2}
\end{figure}

\begin{figure}[h]
	\center{\includegraphics[width=0.95\linewidth]{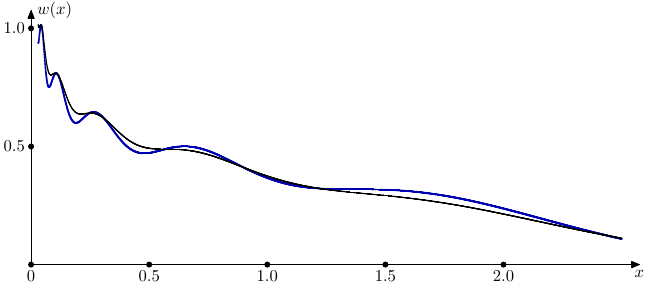}}
	\caption{For the probability generating functions taken from Fig. \ref{fig1}, the densities are compared. Blue curve corresponds to the polynomial $P(z)=(z+z^4)/2$.}\lb{fig2a}
\end{figure}

It's worth noting that oscillations in the limiting characteristics of various types of branching processes have received considerable attention in physical and biological applications, see, e.g., \cite{DIL,DMZ,CG}, as well as in combinatorics, see \cite{Kuz}, \cite{FO}, and \cite{O}. Applied studies are, of course, extremely interesting from different angles of a purely mathematical perspective.

{\bf Remark 2.} The exact calculation of the critical angle $\vt$ can be an interesting problem in its own right. Let us consider an example, partially motivated by \cite{K24a}, with the Poisson-type probability generating function $P(z)$. This function and its inverse are defined by 
\[\lb{b007}
 P(z)=\frac{e^z-1}{e-1},\ \ \ P^{-1}(z)=\ln(1+(e-1)z).
\]
Its Julia set looks like a union of curves, see Fig. \ref{fig3}.(a). It is called the Cantor bouquet. In particular, curves 
\[\lb{b007a}
(P^{-1}+2\pi\mathbf{i}n_t)\circ...\circ (P^{-1}+2\pi\mathbf{i}n_1)([1,+\iy)+2\pi\mathbf{i}n_0), \ \ \ t\in\N\cup\{0\},\ \ \ n_j\in\Z, 
\]
belong to the complement of the open filled Julia set. It seems, see Fig. \ref{fig3}.(b), that the points
\[\lb{b008}
 c_t=\underbrace{P^{-1}\circ...\circ P^{-1}}_{t}(1-2\pi\mathbf{i}),\ \ \ t\in\N,
\]
are the vertices of the piecewise linear envelope of the Julia set in the neighborhood of $z=1$. This observation, in the limit, determines the critical angle
\[\lb{b009}
 \vt=\pi+\lim_{t\to+\iy}\arg (c_t-1).
\]
We have no rigorous proof of \er{b009}. However, if it is true, then we can express $\vt$ through functions already defined above
\[\lb{b010}
 \vt=\pi+\lim_{t\to+\iy}\arg (c_t-1)=\vt=\pi+\lim_{t\to+\iy}\arg E^t(c_t-1)=\pi+\lim_{t\to+\iy}\Im\ln E^t(c_t-1),
\]
or
\[\lb{b011}
 \vt=\pi+\Im\ln(-\Pi^{-1}(1-2\pi\mathbf{i})),
\]
where 
\[\lb{b011a}
 \Pi^{-1}(z)=\lim_{t\to+\iy}E^t(1-\underbrace{P^{-1}\circ...\circ P^{-1}}_{t}(z))
\]
coincides with the inverse of $\Pi$ defined in \er{003}, at least in some neighborhood of $z=1$. It satisfies the Schr\"oder functional equation
\[\lb{b012}
 \Pi^{-1}(P^{-1}(z))=E^{-1}\Pi^{-1}(z).
\]
While we use \er{b011a} directly in our computations, we can also use \er{b012} many times to avoid the questions about the existence of $\Pi^{-1}$ at $z=1-2\pi\mathbf{i}$ in \er{b011}. For example, we can write
\[\lb{b013}
 \vt=\pi+\Im\ln(-\Pi^{-1}(P^{-1}(1-2\pi\mathbf{i}))),
\]
where
\[\lb{b014}
 P^{-1}(1-2\pi\mathbf{i})=\ln(e-2\pi\mathbf{i}(e-1))=1+\ln(1-2\pi\mathbf{i}(1-e^{-1}))
\]
lies about four times closer to $1$.  In general, the strategy for finding the critical angle can be as follows. We construct the Julia set in the neighborhood of $1$ and geometrically determine the points with the smallest argument and a negative imaginary part. These are usually the images of some special points under the action of $P^{-1}$, which can be explicitly computed. Then we may try to prove the corresponding analytical result if we don’t mind the time and energy. 

\begin{figure}[h]
	\centering
	\begin{subfigure}[b]{0.5\textwidth}
		\includegraphics[width=\textwidth]{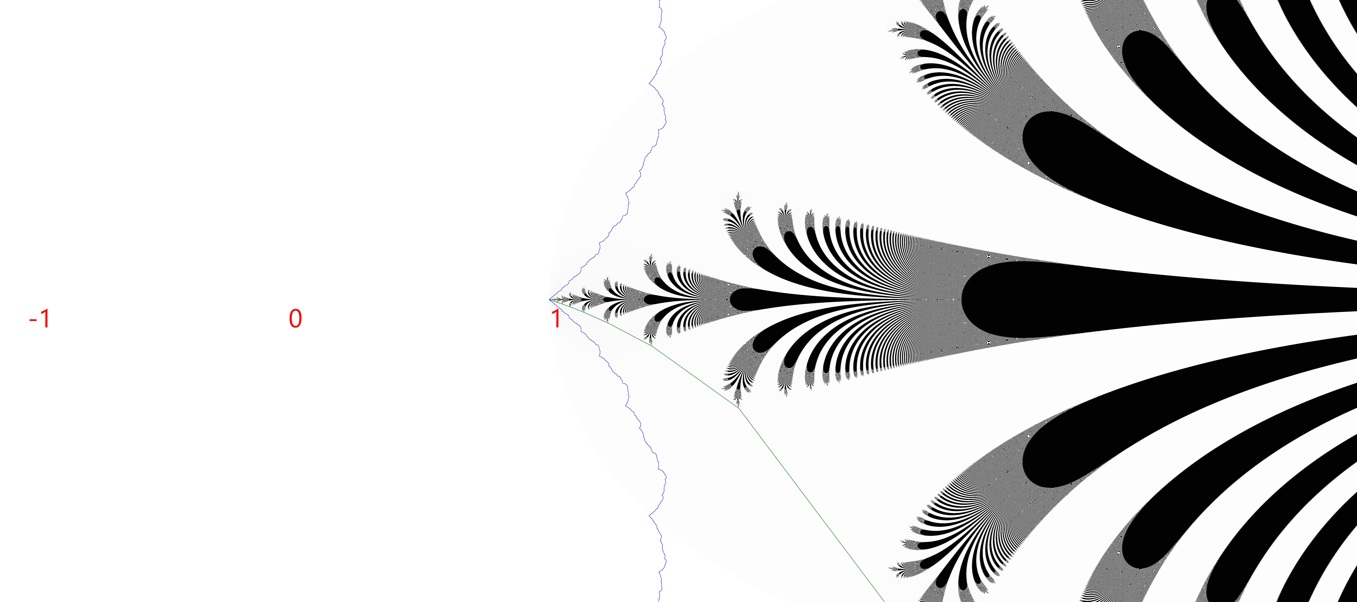}
		\caption{Julia set for \er{b007} (black) and \er{b015} (blue)}
	\end{subfigure}
	\begin{subfigure}[b]{0.49\textwidth}
		\includegraphics[width=\textwidth]{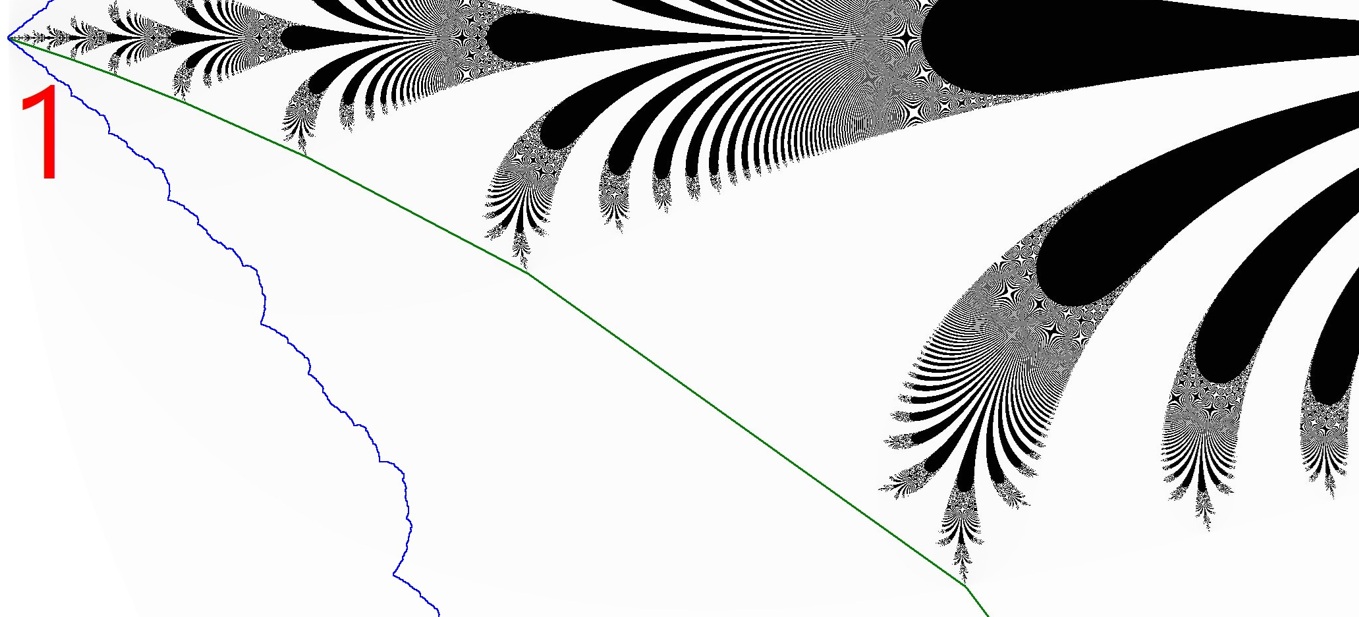}
		\caption{zoomed at $z=1$}
	\end{subfigure}
	\caption{Green line segments connect points $c_t$, see \er{b008}. It seems that, near $z=1$, the Julia set lies above the union of the line segments.}\lb{fig3}. 
\end{figure}

The optimal probability generating function in the class with $p=(e-1)^{-1}$ and $E=e(e-1)^{-1}$ (to which \er{b007} belongs) with empirically small(est?) critical angle is
\[\lb{b015}
 P(z)=\frac{z+(2e-5)z^2+(3-e)z^3}{e-1}.
\]
Indeed, it has a smaller angle, as seen in Fig. \ref{fig3}. By analogy with \er{b008}, one may also verify that the critical angle for \er{b015} corresponds to some preimages of $z=1$, i.e. to
\[\lb{b015a}
 c_t=\underbrace{P^{-1}\circ...\circ P^{-1}}_{t}(1).
\]
Each $P^{-1}$ has three values, and only one of them should be chosen to determine the critical angle in the limit. This study is beyond the scope of this article, but may be useful as a stand-alone exercise for interested readers.

Both critical angles for \er{b007} and \er{b015} are quite large, and the density oscillations are almost invisible (see Fig. \ref{fig4}). The magnitude of oscillations is smaller than $10^{-6}$. Hence, both $V_1(x)$ and $V_2(x)$ are almost constants in this case. Densities themselves are plotted in Fig. \ref{fig5}.

\begin{figure}[h]
	\center{\includegraphics[width=0.95\linewidth]{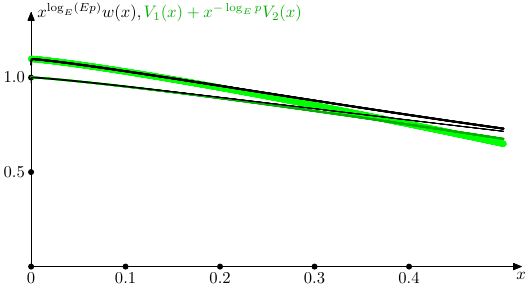}}
	\caption{Normalized densities are compared with the first two normalized terms taken from \er{005}. Larger bold curves correspond to \er{b007}, and smaller thin curves to \er{b015}. }\lb{fig4}
\end{figure}

\begin{figure}[h]
	\center{\includegraphics[width=0.95\linewidth]{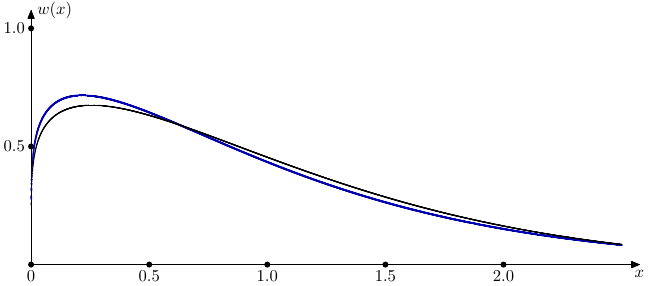}}
	\caption{For the probability generating functions taken from Fig. \ref{fig4}, the densities are compared. Blue curve corresponds to \er{b007}.}\lb{fig5}
\end{figure}

To create all the figures except Fig. \ref{fig1}, I used a very convenient Object Pascal-based programming environment, Delphi Community Edition with the MtxVec and Neslib components, as well as the Metapost interface module for TeXLive and MiKTeX. All Delphi modules for the functions described above were written me. The Fourier transform of $\Pi(\mathbf{i}y)$, see \er{003}, is computed on the interval $y\in[-10^4,10^4]$ using the trapezoidal rule with integration step $10^{-4}$. The same rule and step are used to compute the Fourier coefficients of $K(x)^n$, see \er{006}. To create Figure \ref{fig1}, I turned to Google AI.

{\section{Proof of Theorem \ref{mainT}}\lb{sec1}}

\begin{lemma}\lb{L1}
	Let $\mP_E$ be a set of polynomial probability generating functions with the expectation $P'(1)=E$. Denote $N=\lfloor E\rfloor$ the number which rounds $E$
	down to the nearest integer less than or equal to $E$. Then, the unique minimizer of the second derivative at $z=1$ is
	\[\lb{100}
	 {\rm argmin}_{P\in\mP_E}P''(1)=(N+1-E)z^N+(E-N)z^{N+1}.
	\]
\end{lemma}
{\it Proof.} Without loss of generality, we can assume that all the polynomials have a degree no greater than a sufficiently large $M$. Then the corresponding linear programming problem can be formulated as follows. Minimize 
\[\lb{101}
 Z=\sum_{n=2}^M (n-1)np_n
\]   
subject to
\[\lb{102}
 {\rm all}\ p_n\ge0,\ \ \ \sum_{n=0}^Mp_n=1,\ \ \ \sum_{n=1}^Mnp_n=E.
\]
The solution ${\bf p}^0=(p_n^0)_{n=0}^M$ is reached at some vertex of the compact convex polygon described by \er{102}. Suppose that there are $n_1<n_2<n_3$ such that $p_{n_1}^0\ne0$, $p_{n_2}^0\ne0$, $p_{n_3}^0\ne0$. Denote ${\bf p}^1=(p_{n}^1)_{n=0}^M$, where
\[\lb{103}
 p^1_{n_1}=1,\ \ \ p^1_{n_2}=\frac{n_1-n_3}{n_3-n_2},\ \ \ p^1_{n_3}=\frac{n_2-n_1}{n_3-n_2},\ \ \ {\rm other}\ p^1_n=0.
\] 
Then 
\[\lb{104}
 p^1_{n_1}+p^1_{n_2}+p^1_{n_3}=0,\ \ \ n_1p^1_{n_1}+n_2p^1_{n_2}+n_3p^1_{n_3}=0.
\]
Thus, for all sufficiently small $\ve>0$, ${\bf p}^0+\ve{\bf p}^1$ satisfy \er{102}. Hence, ${\bf p}^0$ is not a vertex. Moreover,
\[\lb{105}
 (n_1-1)n_1p^1_{n_1}+(n_2-1)n_2p^1_{n_2}+(n_3-1)n_3p^1_{n_3}=(n_2-n_1)(n_3-n_1)>0,
\]
and we can decrease $Z({\bf p}^0+\ve{\bf p}^1)$ taking small $\ve<0$. Thus, at most two of $p_n$ are nonzero. Suppose that $p^0_{n_1}\ne0$ and $p^0_{n_3}\ne0$, and others are $0$. Then, solving \er{102} for these two unknowns, we obtain
\[\lb{106}
 p^0_{n_1}=\frac{n_3-E}{n_3-n_1},\ \ \ p^0_{n_3}=\frac{E-n_1}{n_3-n_1},
\]
and, by \er{101},
\[\lb{107}
 Z({\bf p}^0)=(n_3+n_1-1)E-n_3n_1.
\]
If a PGF with mathematical expectation $E$ has all coefficients $p_n^0=0$, except at most $p_{n_1}^0$ and $p_{n_3}^0$, where $n_1\ne n_3$, then it must satisfy \er{106} and \er{107}.

Suppose that $n_3-n_1\ge2$ and $p_{n_1}^0p_{n_3}^0\ne0$. 

The first case, $E-n_1\ge1$. We take $\wt{\bf p}$ with at most two non-zero $\wt p_n$ having indices $n_2=n_1+1$ and $n_3$. By \er{106}, $\wt p_{n_2}\ge0$ ($n_2$ replaces $n_1$) and $\wt p_{n_3}>0$, and, hence, $\wt{\bf p}$ satisfies the constraints \er{102}. Using \er{106} and \er{107}, we obtain
\[\lb{108}
 Z({\bf p}^0)-Z(\wt{\bf p})=(n_1-n_2)(E-n_3)=n_3-E=(n_3-n_1)p^0_{n_1}>0.
\]
This contradicts the minimality of $Z({\bf p}^0)$. 

The second case, $E-n_1<1$. Since $n_3-n_1\ge2$, we have $n_3-E>1$. Let us take $\wt{\bf p}$ with at most two non-zero $\wt p_n$ having indices $n_1$ and $n_2=n_3-1$. By \er{106}, $\wt p_{n_2}\ge0$ ($n_2$ replaces $n_3$) and $\wt p_{n_1}>0$, and, hence, $\wt{\bf p}$ satisfies the constraints \er{102}. Using \er{106} and \er{107}, we obtain
\[\lb{109}
 Z({\bf p}^0)-Z(\wt{\bf p})=(n_3-n_2)(E-n_1)=E-n_1=(n_3-n_1)p^0_{n_3}>0.
\]
Again, this contradicts the minimality of $Z({\bf p}^0)$. 

Thus, $n_3-n_1\ge2$ is impossible. Hence, $n_3=n_1+1$ or only one among $p_{n_1}^0$ and $p_{n_3}^0$ is non-zero. In both cases, the RHS of \er{100} is the corresponding unique minimizer. \BBox

\begin{lemma}\lb{L2}
	Let $\mP_{p,E}$ be the set of polynomial probability generating functions with the expectation $P'(1)=E$ and, also, satisfying $P(0)=0$, $P'(0)=p$. Denote $N=\lfloor \frac{E-p}{1-p}\rfloor$. Then, the unique minimizer of the second derivative at $z=1$ is
	\[\lb{110}
	{\rm argmin}_{P\in\mP_{p,E}}P''(1)=pz+(N(1-p)+1-E)z^N+(E-p-N(1-p))z^{N+1}.
	\]
It is assumed that $E\ge 2-p$, because, otherwise, $\mP_{p,E}=\varnothing$, and \er{110} does not make sense.
\end{lemma}
{\it Proof.} The slightly modified Proof of Lemma \ref{L2} can be repeated. \BBox

\begin{corollary}\lb{C0} The same minimizer as in \er{110} solves also ${\rm argmin}_{P\in\mP_{p,E}}\deg P$.
\end{corollary}
{\it Proof.} It follows from the properties of the solution in Lemma \ref{L2}, i.e. from $E=p+2p_2+...+(\deg P)p_{\deg P}\le p+(1-p)\deg P$, which gives $\deg P\ge (E-p)/(1-p)$. \BBox

\begin{theorem}\lb{T1} Suppose that the probability generating function $P(z)$ is analytic in some neighborhood of $z=1$, and $E=P'(1)>1$. Then the unique solution of the Poincar\'e functional equation 
\[\lb{111}
 \Pi(z)=P(\Pi(\frac{z}E)),\ \ \ \Pi(0)=1,\ \ \ \Pi'(0)=-1
\]
exists in some neighborhood of $z=0$. It can be computed by
\[\lb{111a}
 \Pi(z)=\lim_{t\to+\iy}\underbrace{P\circ...\circ P}_{t}(1-\frac{z}{E^{t}}).
\]
\end{theorem}
{\it Proof.} We have
\[\lb{112}
 P(1+z)=1+Ez+z^2Q(z),\ \ \ |z|\le\d_0
\]
with some bounded analytic function $Q(z)$ satisfying $|Q(z)|\le C$ for $|z|\le\d_0$. Here, $\d_0>0$ and $C>0$ are some constant. Instead of $\Pi(z)$, we are looking for $\Psi(z)$ defined by
\[\lb{113}
 \Pi(z)=1-z+z^2\Psi(z).
\]
Substituting \er{113} into \er{111} and using \er{112}, we obtain
\[\lb{114}
 \Psi(z)=\frac1E\Psi\lt(\frac zE\rt)+\frac1{E^2}\lt(1-\frac zE\Psi\lt(\frac zE\rt)\rt)^2Q\lt(-\frac zE+\frac{z^2}{E^2}\Psi\lt(\frac zE\rt)\rt)=:\cT(\Psi)(z).
\]
Denote
\[\lb{115}
 R=4C\frac{E-1}{E},\ \ \ \d_1=\min\lt\{\frac{E}{R},\frac{\d_0 E}2\rt\}.
\]
The parameters in \er{115} are chosen such that if $|z|\le\d_1$ then
\[\lb{116}
 \lt|-\frac zE+\frac{z^2}{E^2}\Psi\lt(\frac zE\rt)\rt|\le\frac{\d_1}{E}+\frac{R\d_1^2}{E^2}\le\frac{\d_0}2+\frac{R\d_1}{E}\cdot\frac{\d_1}{E}\le\frac{\d_0}2+\frac{\d_0}2\le\d_0.
\]
Hence, $Q$ in \er{114} is less than $C$, see below \er{112}, which gives
\[\lb{117}
 |\cT(S)(z)|\le\frac{R}{E}+\frac C{E^2}\lt(1+\frac{R\d_1}{E}\rt)^2\le\frac{R}{E}+\frac{4C}{E^2}=\frac{R}E+\frac{R}{E(E-1)}=R,
\]
for any analytic function $S(z)$ satisfying $|S(z)|\le R$ when $|z|\le\d_1$.
Thus $\cT:\mathbb{B}_{\mathbb{H}^{\iy}_{\d}}(R)\to\mathbb{B}_{\mathbb{H}^{\iy}_{\d}}(R)$ acts on the ball of radius $R$ in the corresponding Hardy space of analytic functions defined on $\mathbb{S}_{\d}=\{z:\ |z|<\d\}$, for any $\d\le\d_1$. Moreover, since $E>1$, the main term $\Psi(z/E)/E$ in \er{114} is a contraction mapping, and the second term can be made arbitrarily small by choosing a small $\d$, because $|z|<\d$. Thus, $\cT:\mathbb{B}_{\mathbb{H}^{\iy}_{\d}}(R)\to\mathbb{B}_{\mathbb{H}^{\iy}_{\d}}(R)$ is a contractive mapping for all sufficiently small $\d>0$. By the Banach Fixed-Point Theorem, there is a unique solution of \er{114}, which, with \er{113}, gives the analytic solution of \er{112}. As a fixed point solution of \er{111}, $\Pi(z)$ can be computed by \er{111a}. \BBox

\begin{corollary}\lb{C1}
 Under assumptions of Theorem \ref{T1}, suppose also that 
\[\lb{c001}
 1-\{re^{i\vp}:\ \vp\in(-\vp_0,\vp_0),\ r\in(0,r_0)\}\ss\mJ_P(0).
\]
Then $\Pi(z)$ can be analytically extended to the sector $S_{\vp_0}:=\{z:\ |\arg z|<\vp_0\}$. The image of this sector satisfies $\Pi(S_{\vp_0})\ss\mJ_P(0)\cup\{1\}$.
\end{corollary}
{\it Proof.} If $|\arg z|<\vp_0$ and $z\ne0$ then, for any sufficiently small $\ve>0$, $\Pi(\ve z)=1-\ve z+O(\ve^2)$ belongs to $\mJ_P(0)$ by \er{c001}. Thus, for any $n\in\N$, the value 
\[\lb{c001a}
 \Pi(E^n\ve z)=\underbrace{P\circ...\circ P}_{n}(\Pi(\ve z))
\] 
is well defined by \er{111} and the fact that $P(\mJ_P(0))\ss\mJ_P(0)$. \BBox

\begin{theorem}\lb{T3}
	Let $P_0(z)$ be a probability generating function analytic in some neighborhood of $z=1$ and satisfying  $P_0(z)=pz+...$ with $0<p<1$. Then $\Pi_0(z)$, the solution of \er{111}, can be analytically extended to a sector $S_{\vp_0}$, see definition in Corollary \ref{C1}, with some $\vp_0>\pi/2$. Moreover, $\Pi_0(z)\to0$ uniformly for $z\to\iy$ and $|\arg z|\le\vp_0$.
\end{theorem}
{\it Proof.}  By Lemma \ref{L1} we have
\[\lb{T001}
 P''_0(1)\ge\min_{P\in\mP_E}P''(1)=E(E-1)+(E-N)(E-N+1).
\]
Thus $P''_0(1)\ge E^2-E$ and the equality is reached only if $P_0(z)=z^N$. In our case, $P_0(z)\ne z^N$, and, hence, $P_0''(1)>E^2-E$, where $E=P_0'(1)\ge2-p>1$. Identity \er{111} allows us to compute $\Pi_0''(0)$, namely $\Pi_0''(0)=P_0''(1)/(E^2-E)>1$. Thus, for any sufficiently small real $y\ne0$,
\[\lb{T002}
 \Pi_0(\mathbf{i}y)=1-\mathbf{i}y-\frac{\Pi_0''(0)y^2}2+O(y^3)\in\mathbb{S}_1,
\]
since
\[\lb{T003}
 |\Pi_0(\mathbf{i}y)|^2=1+(1-\Pi_0''(0))y^2+O(y^3)<1.
\]
Similarly, for both sufficiently small $x\ge0$ and $y\ne0$
\[\lb{T004}
 \Pi_0(x+\mathbf{i}y)\in\mathbb{S}_1.
\]
Thus, due to \er{T004} and continuity of $\Pi_0$, there are $\vp_0>\pi/2$, sufficiently small $r_0>0$, and some $1>\d>0$ such that
\[\lb{T005}
 \Pi_0(\{re^{\mathbf{i}\vp}:\ r\in[r_0,Er_0],\ \vp\in[-\vp_0,\vp_0]\})\ss\mathbb{S}_{\d}.
\]
Formulas \er{111} and \er{121} lead to
\[\lb{T006}
\Pi_0(\{re^{\mathbf{i}\vp}:\ r\in[E^nr_0,E^{n+1}r_0],\ \vp\in[-\vp_0,\vp_0]\})\ss\underbrace{P\circ...\circ P}_{n}(\mathbb{S}_{\d})\ss\mathbb{S}_{\d(p+(1-p)\d)^n},
\]
which gives the required uniform convergence to $0$ in the sector, since $\d(p+(1-p)\d)^n\to0$ for $n\to\iy$. \BBox

\begin{corollary}\lb{C2}
	If PGF $P(z)$ satisfies $P(0)=0$ and $0<P'(0)<1$ then the critical angle $\vt>\pi/2$.
\end{corollary}

\begin{theorem}\lb{T4} Suppose that the probability generating function $P(z)$ satisfies $P(0)=0$ and $0<p:=P'(0)<1$. Then, the unique solution of the Shr\"oder functional equation 
\[\lb{118}
 \Phi(P(z))=p\Phi(z),\ \ \ \Phi(0)=0,\ \ \ \Phi'(0)=1
\]
exists in some neighborhood of $z=0$. It can be extended analytically to the open connected component of the filled Julia set $\cJ_P(0)$. For its computation, one may use
\[\lb{118a}
\Phi(z)=\lim_{t\to+\iy}p^{-t}\underbrace{P\circ...\circ P}_{t}(z).
\]
\end{theorem}
{\it Proof.} Instead of $\Phi(z)$, we are looking for $\Psi(z)$ defined by
\[\lb{119}
 \Phi(z)=z+z^2\Psi(z).
\]
Substituting \er{113} into \er{111} and using \er{112}, we obtain
\[\lb{120}
 \Psi(z)=\frac{P(z)-pz}{pz^2}+\frac{P(z)^2}{pz^2}\Psi(P(z))=:\frac{P(z)-pz}{pz^2}+\cT(\Psi)(z).
\]
The linear mapping $\cT:\mathbb{H}^{\iy}_{\d}\to\mathbb{H}^{\iy}_{\d}$ acts on the corresponding Hardy space of analytic functions defined in the disc $\mathbb{S}_{\d}=\{z:\ |z|<\d\}$, for any $\d<1$. This is because $P(z)$ is a contraction mapping, i.e. 
\[\lb{121}
 |P(z)|\le p|z|+p_2|z|^2+...\le p|z|+(1-p)|z|^2\le\d(p+(1-p)\d)<\d\ {\rm for}\ |z|\le\d<1.
\]
Moreover, for $|z|\le\d$, we have
\[\lb{122}
 \lt|\frac{P(z)^2}{pz^2}\rt|\le{p\lt(1+\d\frac{1-p}{p}\rt)}<1,\ {\rm for\ all\ sufficiently\ small}\ \d>0.
\]
Hence, $\|\cT\|_{\mathbb{H}^{\iy}_{\d}\to\mathbb{H}^{\iy}_{\d}}<1$ for such small $\d>0$, and the unique solution of \er{120} exists
\[\lb{123}
 \Psi(z)=(1-\cT)^{-1}\frac{P(z)-pz}{pz^2}.
\]
This formula can be viewed as an application of the fixed-point idea. Hence, as a fixed point solution of \er{118}, $\Phi(z)$ can be computed by \er{118a}. Note that the function in the RHS of \er{123} is analytic at $z=0$, since $P(z)\sim pz$ for $z\to0$. By $\Phi(z)=\Phi(P(z))/p$, we can extend this solution analytically to  
\[\lb{124}
 z\in\bigcup_{n=1}^{+\iy}\underbrace{P^{-1}\circ...\circ P^{-1}}_{n}(\mathbb{S}_{\d}).
\]
This union contains the corresponding open connected component of the filled Julia set. \BBox

\no{\bf Proof  of Theorem \ref{mainT}.} We take $\vt_1=\vt-\ve$ for some small $\ve>0$ so that $\vt_1>\pi/2$. By Corollary \ref{C1}, if $|\Im z|\le\vt_1/\ln E$ then $\Pi(E^z)\in\mJ_P(0)$ and the function
\[\lb{125}
K(z)=z^{-p}\Phi(\Pi(E^z))
\]
is defined and analytic. Moreover, it satisfies $K(z+1)=K(z)$ and, hence, can be expanded into the Fourier series
\[\lb{126}
 K(z)=\sum_{m=-\iy}^{m=+\iy}\k_me^{2\pi\mathbf{i}mz}.
\]
One may compute the coefficients by
\[\lb{127}
 \k_m=\int_{[0,1]+\sign(m)\mathbf{i}\frac{\vt_1}{\ln E}}e^{-2\pi\mathbf{i}mz}K(z)dz,
\]
where we assume $\sign(0)=0$. Denoting $K_0=\max|K(z)|$ in the rectangle $z\in[0,1]\ts[-\vt_1/\ln E,\vt_1/\ln E]$ and using \er{127}, we obtain the estimates
\[\lb{128}
 |\k_m|\le K_0e^{-|m|\frac{\vt_1}{\ln E}}.
\] 
Generally, for the one-periodic functions
\[\lb{129}
K(z)^n=\sum_{m=-\iy}^{m=+\iy}\k_m^{*n}e^{2\pi\mathbf{i}mz},\ \ \ n\in\N,
\]
we have
\[\lb{130}
|\k_m^{*n}|\le K_0^ne^{-|m|\frac{2\pi\vt_1}{\ln E}}.
\] 
Identity \er{125} can be rewritten as
\[\lb{131}
 \Phi(\Pi(z))=z^{\log_Ep}K(\log_Ez),\ \ \ |\arg z|\le\vt_1.
\]
Since $\Pi(z)\to0$ uniformly for $z\to\iy$ in the sector $|\arg z|\le\vt_1$, we can rewrite \er{131} as
\[\lb{132}
 \Pi(z)=\sum_{n=1}^{+\iy}\phi_nz^{n\log_Ep}K(\log_Ez)^n,\ \ \ |\arg z|\le\vt_1,\ |z|\ge R_2,
\]
with some large $R_2>0$. Here,
\[\lb{133}
 \Phi^{-1}(z)=\sum_{n=1}^{+\iy}\phi_nz^n
\]
is analytic in some small neighborhood of $z=0$, where $\Phi(z)\sim z$. Due to the Cauchy estimates, we have
\[\lb{134}
 |\phi_n|\le C_3R_3^n,
\]
for some constants $C_3,R_3>0$. Since, for $x>0$, $e^{zx}$ is bounded along lines parallel to the imaginary axis, and $\Pi(z)$ uniformly tends to $0$ for large $z$ in the sector $|\arg z|\le\vt_1$, we can write
\[\lb{135}
 w(x)=\frac1{2\pi\mathbf{i}}\int_{\mathbf{i}\R}\Pi(z)e^{zx}dz=\frac1{2\pi\mathbf{i}}\int_{a+\mathbf{i}\R}\Pi(z)e^{zx}dz,
\]
for any $a>0$. The existence of both integrals - in fact, the existence of one of them leads, obviously, to the existence of the second one - will be discussed below. Taking $a$ so large that \er{132} is valid and substituting it into \er{135}, we get
\[\lb{136}
w(x)=\frac1{2\pi\mathbf{i}}\int_{a+\mathbf{i}\R}\sum_{n=1}^{+\iy}\phi_nz^{n\log_Ep}K(\log_Ez)^ne^{zx}dz
\]
For any sufficiently large $a>0$, the tail in \er{136} is uniformly small and converges absolutely, since $n\log_Ep<-1$ for all large $n$, and we can estimate the tail by the tail of a converging geometric series. For the moment, denote $\wt K_n(w):=K(w)^n$. Suppose $\g:=n\log_Ep\ge-1$. Still $\g<0$ because $p<1$, and we can apply $m$ times the integration by parts
\[\lb{137}
 \int_{a+\mathbf{i}\R}z^{\g}\wt K_n(\log_Ez)e^{zx}dz=\int_{a+\mathbf{i}\R}z^{\g-1}(x^{-1}\wt K_n(\log_Ez)-E^{-1}\wt K_n'(\log_Ez))e^{zx}dx,
\]
to make the integral absolutely convergent with $\g-m<-1$. We can apply this integration by parts many times since the multiplier in the brackets is $\log$-periodic as initial $\wt K_n(\log_Ez)$. 
Thus, the integrals of the first terms in \er{136} exist, and we can write
\[\lb{138}
w(x)=\sum_{n=1}^{+\iy}\frac{\phi_n}{2\pi\mathbf{i}}\int_{a+\mathbf{i}\R}z^{n\log_Ep}K(\log_Ez)^ne^{zx}dz.
\]
On the line $a+\mathbf{i}\R$, $|\Im(\log_Ez)|\le\pi/(2E)<\vt_1/E$, and we can use expansion \er{129}, which, by \er{130}, converges in $m$ exponentially fast, in the integrals appearing in \er{138}, namely
\begin{multline}\lb{139}
 \frac{1}{2\pi\mathbf{i}}\int_{a+\mathbf{i}\R}z^{n\log_Ep}K(\log_Ez)^ne^{zx}dz=\frac{1}{2\pi\mathbf{i}}\int_{a+\mathbf{i}\R}\sum_{m=-\iy}^{m=+\iy}\k_m^{* n}z^{\frac{n\ln p+2\pi\mathbf{i}m}{\ln E}}e^{zx}dz\\
 =\sum_{m=-\iy}^{m=+\iy}\frac{\k_m^{* n}}{2\pi\mathbf{i}}\int_{a+\mathbf{i}\R}z^{\frac{n\ln p+2\pi\mathbf{i}m}{\ln E}}e^{zx}dz=\sum_{m=-\iy}^{m=+\iy}\frac{\k_m^{*n}x^{-\frac{n\ln p+2\pi\mathbf{i}m}{\ln E}}}{x\G(-\frac{n\ln p+2\pi\mathbf{i}m}{\ln E})},
\end{multline}
where the Hankel representation of the Gamma function, see Example 12.2.6 on page 254 of \cite{WW}, is used. Substitution of \er{139} into \er{138} gives the double series. Let us show that this double series converges absolutely. We estimate the logarithms of its terms. For any fixed $\d>0$, the Gamma function satisfies 
\[\lb{140}
 \ln\G(z)=(z-\frac12)\ln z-z+O(1),\ \ \ \Re z\ge\d.
\] 
Thus, for all $n\in\N$ and $m\in\Z$, we obtain
\[\lb{141}
 \Re\ln\G(-\frac{n\ln p+2\pi\mathbf{i}m}{\ln E})\ge -n\log_Ep\cdot\ln(-n\log_Ep)+n\log_Ep-\frac{\pi^2|m|}{\ln E}-C_4,
\]
with some constant $C_4>0$. Estimate \er{141} along with \er{130} and \er{134} gives
\[\lb{142}
 \lt|\frac{\phi_n\k_m^{*n}x^{-\frac{n\ln p+2\pi\mathbf{i}m}{\ln E}}}{\G(-\frac{n\ln p+2\pi\mathbf{i}m}{\ln E})}\rt|\le Ae^{({\log_Ep}+\ve)n\ln n+\frac{\pi(\pi-2\vt_1)}{\ln E}|m|},\ \ \ \forall\ve>0,
\]
where $A=A(\ve, x)>0$ is uniformly bounded for $x$ belonging to any compact subset of $(0,+\iy)$. Note that we split off $-\ve n\ln n$ in \er{142} to compensate for the terms of order $n$ under the exponent. Here, for convenience, $\ve>0$ can be the same as in $\vt_1=\vt-\ve$ taken at the beginning. \BBox

\section*{Acknowledgements} 
This paper is a contribution to the project S1 of the Collaborative Research Centre TRR 181 "Energy Transfer in Atmosphere and Ocean" funded by the Deutsche Forschungsgemeinschaft (DFG, German Research Foundation) - Projektnummer 274762653. 



\bibliographystyle{abbrv}

\end{document}